\documentclass[a4,12pt]{amsart}
\usepackage{amssymb,amstext,amsmath,amscd,amsthm,amsfonts,enumerate,latexsym}
\usepackage{color}
\usepackage[dvipdfmx]{graphicx}
\usepackage[all]{xy}
\theoremstyle{plain}
\newtheorem{thm}{Theorem}[section]
\newtheorem*{thm*}{Theorem}
\newtheorem*{cor*}{Corollary}

\newtheorem{prop}[thm]{Proposition}
\newtheorem{lem}[thm]{Lemma}
\newtheorem{cor}[thm]{Corollary}

\newtheorem*{claim*}{Claim}

\theoremstyle{definition}
\newtheorem{defn}[thm]{Definition}
\newtheorem{ex}[thm]{Example}

\theoremstyle{remark}

\numberwithin{equation}{thm}

\def\Im{\operatorname{Im}}
\def\Ker{\operatorname{Ker}}

\def\Hom{\operatorname{Hom}}

\def\mod{\mathrm{mod}}

\def\rank{\mathrm{rank}}

\def\e{\mathrm{e}}
\def\m{\mathfrak m}
\def\n{\mathfrak n}

\def\p{\mathfrak p}

\def\K{\mathrm{K}}

\newcommand{\rma}{\mathrm{a}}

\newcommand{\rmr}{\mathrm{r}}

\newcommand{\rmK}{\mathrm{K}}

\newcommand{\rmQ}{\mathrm{Q}}

\newcommand{\calR}{\mathcal{R}}
\newcommand{\calS}{\mathcal{S}}

\newcommand{\fkn}{\mathfrak{n}}

\newcommand{\fkp}{\mathfrak{p}}

\newcommand{\fkM}{\mathfrak{M}}
\newcommand{\fkN}{\mathfrak{N}}

\newcommand{\mapright}[1]{%
\smash{\mathop{%
\hbox to 1cm{\rightarrowfill}}\limits^{#1}}}

\newcommand{\mapleft}[1]{%
\smash{\mathop{%
\hbox to 1cm{\leftarrowfill}}\limits_{#1}}}

\def\Sym{\mathrm{Sym}}

\def\R{{\mathcal R}}

\def\R{{R}}

\begin{document}

\setlength{\baselineskip}{15pt}
\title{The almost Gorenstein  Rees algebras over two-dimensional regular local rings}
\author{Shiro Goto}
\address{Department of Mathematics, School of Science and Technology, Meiji University, 1-1-1 Higashi-mita, Tama-ku, Kawasaki 214-8571, Japan}
\email{goto@math.meiji.ac.jp}
\author{Naoyuki Matsuoka}
\address{Department of Mathematics, School of Science and Technology, Meiji University, 1-1-1 Higashi-mita, Tama-ku, Kawasaki 214-8571, Japan}
\email{naomatsu@meiji.ac.jp}
\author{Naoki Taniguchi}
\address{Department of Mathematics, School of Science and Technology, Meiji University, 1-1-1 Higashi-mita, Tama-ku, Kawasaki 214-8571, Japan}
\email{taniguti@math.meiji.ac.jp}
\author{Ken-ichi Yoshida}
\address{Department of Mathematics, College of Humanities and Sciences, Nihon University, 3-25-40 Sakurajosui, Setagaya-Ku, Tokyo 156-8550, Japan}
\email{yoshida@math.chs.nihon-u.ac.jp}

\thanks{The first author was partially supported by JSPS Grant-in-Aid for Scientific Research 25400051. 
The second author was partially supported by JSPS Grant-in-Aid for Scientific Research 26400054. The third author was partially supported by Grant-in-Aid for JSPS Fellows 26-126 and by JSPS Research Fellow. The fourth author was partially supported by JSPS Grant-in-Aid for Scientific Research 25400050.}

\begin{abstract} Let $(R,\m)$ be a two-dimensional regular local ring with infinite residue class field. Then the Rees algebra $\calR (I)= \bigoplus_{n \ge 0}I^n$ of $I$ is an almost Gorenstein graded ring in the sense of \cite{GTT}  for every $\m$-primary integrally closed ideal $I$ in $R$.
\end{abstract}

\maketitle

\section{Introduction}\label{intro}
The main purpose of this paper is to prove the following.

\begin{thm}\label{3.4}
Let $(R,\m)$ be a two-dimensional regular local ring with infinite residue class field and $I$ an $\m$-primary integrally closed ideal in $R$. Then the Rees algebra $\calR(I)=\bigoplus_{n \ge 0}I^n$ of $I$ is an almost Gorenstein graded ring. 
\end{thm}

As a direct consequence one has the following.

\begin{cor}\label{3.5}
Let $(R,\m)$ be a two-dimensional regular local ring with infinite residue class field. Then $\calR(\m^\ell)$  is an almost Gorenstein graded ring for every integer $\ell > 0$.
\end{cor}

The proof of Theorem 1.1 depends on a result of J. Verma \cite{V} which guarantees the existence of joint reductions with joint reduction number zero. Therefore our method of proof works also for two-dimensional rational singularities, which we shall discuss in the forthcoming paper \cite{GTY1}. Here, before entering details, let us recall the notion of almost Gorenstein {\it graded/local} ring as well as some historical notes  about it.

Almost Gorenstein rings are a new class of Cohen-Macaulay rings, which are not necessarily Gorenstein, but still good, possibly next to the Gorenstein rings. The notion of these local rings dates back to the paper \cite{BF}  of V. Barucci and R. Fr\"oberg in 1997, where they dealt with one-dimensional analytically unramified local rings.  Because their notion is not flexible enough to analyze analytically ramified rings, in 2013 S. Goto, N. Matsuoka, and T. T. Phuong \cite{GMP} extended the notion to arbitrary (but still of dimension one)  Cohen-Macaulay local rings. The reader may consult \cite{GMP} for examples of analytically ramified almost Gorenstein local rings. In 2015 S. Goto, R. Takahashi, and N. Taniguchi \cite{GTT} finally gave  the definition of almost Gorenstein graded/local rings of higher dimension. In \cite{GTT} the authors developed a very nice theory, exploring many concrete examples. We recall here the definition of almost Gorenstein graded rings given in \cite{GTT}. We refer to \cite[Definition 1.1]{GTT} for the definition of almost Gorenstein {\it local} rings.

\begin{defn}[{\cite[Definition 1.5]{GTT}}]\label{1.2} Let $R = \bigoplus_{n \ge 0}R_n$ be a Cohen-Macaulay graded ring such that   $R_0$ is a local ring. Suppose that $R$ possesses the graded canonical module $\K_R$. Let $\fkM$ be the unique graded maximal ideal of $R$ and $a = \rma (R)$ the $\rma$-invariant of $R$. Then we say that $R$ is an almost Gorenstein {\it graded} ring, if there exists an exact sequence
$$
0 \to R \to \mathrm{K}_R(-a) \to C \to 0
$$
of graded $R$-modules such that  $\mu_R(C) = \e_\fkM^0(C)$, where $\mu_R(C)$ denotes the number of elements in a minimal system of generators of $C$ and $\e_\fkM^0(C)$ is the multiplicity of $C$ with respect to $\fkM$. Remember that  $\mathrm{K}_R(-a)$ stands for   the graded $R$-module whose underlying $R$-module is the same as that of $\K_R$ and whose grading is given by $[\mathrm{K}_R(-a)]_n = [\mathrm{K}_R]_{n-a}$ for all $n \in \Bbb Z$. 
\end{defn} 

Definition \ref{1.2} means that if $R$ is an almost Gorenstein graded ring, then even though  $R$ is not a Gorenstein ring, $R$ can be  embedded into the graded $R$-module $\mathrm{K}_R(-a)$, so that the difference $\mathrm{K}_R(-a)/R$ is a graded  Ulrich $R$-module (\cite{BHU}) and behaves well. The reader may consult \cite{GTT} about the basic theory of almost Gorenstein graded/local  rings and the relation between the graded theory and the local theory, as well.  For instance, it is shown  in \cite{GTT} that certain Cohen-Macaulay local rings of finite Cohen-Macaulay representation type, including two-dimensional rational singularities, are almost Gorenstein local rings. The almost Gorenstein local rings which are not Gorenstein are G-regular (\cite[Corollary 4.5]{GTT}) in the sense of \cite{T}. They are now getting revealed to enjoy good properties.

The theory of Rees algebras has been satisfactorily developed and nowadays one knows many Cohen-Macaulay Rees algebras (see, e.g., \cite{GS, HU, MU, SUV}). Among them Gorenstein Rees algebras are rather rare (\cite{Ikeda}). Nevertheless, although  they are not Gorenstein, some of Cohen-Macaulay Rees algebras are still good and could be {\it almost Gorenstein} rings, which we are eager to report in this paper. Except \cite[Theorems 8.2, 8.3]{GTT}, our Theorem 1.1 is the first attempt to answer the question of when the Rees algebras are almost Gorenstein graded rings. For these reasons our Theorem 1.1 might have some significance.

We briefly explain how this paper is organized. The proof of Theorem 1.1 shall  be given in Section 2. For the Rees algebras of modules over two-dimensional regular local rings  we have a similar result, which we give in Section 2 (Corollary \ref{3.6}). 
In Section 3 we explore the case where the ideals are linearly presented over power series rings. The result (Theorem \ref{4.1}) seems to  suggest that almost Gorenstein Rees algebras are still rare, when the dimension of base rings is greater than two, which we shall discuss in the forthcoming paper \cite{GTY2}. In Section 4 we study the Rees algebras of socle ideals $Q:\m$ in a two-dimensional regular local ring $(R,\m)$ and show that Rees algebras are not necessarily almost Gorenstein graded rings even for these ideals (Corollary  \ref{badex}).

In what follows, unless otherwise specified, let $(R,\m)$ denote a Cohen-Macaulay local ring. For each finitely generated $R$-module $M$ let $\mu_R(M)$ (resp. $\e_\m^0(M)$) denote the number of elements in a minimal system of generators for $M$ (resp. the multiplicity of $M$ with respect to $\m$). Let $\mathrm{K}_R$ stand for the canonical module of $R$.




\section{Proof of Theorem 1.1}
The purpose of this section is to prove Theorem 1.1. Let $(R, \m)$ be a Gorenstein local ring with $\dim R = 2$ and let $I \subsetneq R$ be an $\m$-primary ideal of $R$. Assume that $I$ contains  a parameter ideal $Q=(a,b)$ of $R$ such that $I^2 = Q I$. We set $J=Q:I$. Let
$$
\calR = R[It] \subseteq R[t] \ \ \text{and} \ \ T = R[Qt]\subseteq R[t],
$$
where $t$ stands for an indeterminate over $R$.  Remember that the Rees algebra $\calR$ of $I$ is a Cohen-Macaulay ring (\cite{GS}) with $\rma(\calR) = -1$ and $\calR = T + T{\cdot}It$, while  the Rees algebra $T$ of $Q$ is a Gorenstein ring of dimension $3$  and $\mathrm{a}(T)=-1$ (remember that $T \cong R[x,y]/(bx-ay)$). Hence $\rmK_T(1) \cong T$ as a graded $T$-module, where $\mathrm{K}_T$ denotes the graded canonical module of $T$.

Let us begin with the following, which  is a special case of \cite[Theorem 2.7 (a)]{U}. We note a brief proof.

\begin{prop}\label{3.1}
$\rmK_{\calR}(1) \cong J \calR$ as a graded $\calR$-module.
\end{prop}

\begin{proof}
Since $\calR$ is a module-finite extension of $T$, we get
$$
\rmK_R(1) \cong \Hom_T(\calR, \rmK_T)(1) \cong \Hom_T(\calR, T) \cong T:_F \calR
$$
as graded $\calR$-modules, where $F=\rmQ(T) =\rmQ(\calR)$ is the total ring of fractions. Therefore $T:_F \calR = T:_T It$, since $\calR = T + T{\cdot}It$. Because $Q^n \cap [Q^{n+1}:I] = Q^n[Q:I]$ for all $n \ge 0$, we have $T:_T It = JT$. Hence $T:_F \calR=JT$, so that $JT = J\calR$. Thus  $\rmK_{\calR}(1) \cong J \calR$ as a graded $\calR$-module.
\end{proof}

\begin{cor}\label{3.2}
The ideal $J =Q:I$ in $R$ is  integrally closed, if $\calR$ is a normal ring.
\end{cor}

\begin{proof}
Since $J\calR \cong \mathrm{K}_\calR(1)$, the ideal $J\calR$ of $\calR$ is unmixed and of height one. Therefore, if $\calR$ is a normal ring, $J\calR$ must be  integrally closed in $\calR$, whence $J$ is integrally closed in $R$ because $\overline{J} \subseteq J\calR$, where $\overline{J}$ denotes the integral closure of $J$.
\end{proof}

Let us give the following criterion for $\calR$ to be a special kind of almost Gorenstein graded rings. Notice that Condition (2) in Theorem \ref{3.3} requires the existence of joint reductions of $\m$, $I$, and $J$ with reduction number zero (cf. \cite{V}).

\begin{thm}\label{3.3} With the same notation as above, set $\fkM = \m\calR + \calR_+$, the graded maximal ideal of $\calR$. Then the following conditions are equivalent.
\begin{enumerate}[$(1)$]
\item There exists an exact sequence
$$
0 \to \calR \to \rmK_{\calR}(1) \to C \to 0
$$
of graded $\calR$-modules such that $\fkM C = (\xi, \eta)C$ for some homogeneous elements $\xi, \eta$ of $\fkM$.
\item There exist elements $f\in \m$, $g \in I$, and $h \in J$ such that
$$
IJ = gJ + I h \ \ \ \text{and} \ \ \ \m J = f J + \m h.
$$
\end{enumerate}
When this is the case, $\calR$ is an almost Gorenstein graded ring. 
\end{thm}

\begin{proof}
$(2) \Rightarrow (1)$  Notice that $\fkM{\cdot}J\calR \subseteq (f, gt){\cdot}J\calR + \calR h$, since $IJ = gJ + I h$ and $\m J = f J + \m h$. Consider the exact sequence
$$
\calR \overset{\varphi}{\to} J\calR \to C \to 0
$$
of graded $\calR$-modules where $\varphi(1) = h$. We then have $\fkM C = (f, gt)C$, so that $\dim_{\calR_{\fkM}}C_{\fkM} \le 2$. Hence by \cite[Lemma 3.1]{GTT} the homomorphism $\varphi$ is injective and $\calR$ is an almost Gorenstein graded ring.

$(1) \Rightarrow (2)$  Suppose that $\calR$ is a Gorenstein ring. Then  $\mu_R(J) = 1$, since $\mathrm{K}_\calR(1)\cong J\calR$. Hence $J = R$ as $\m \subseteq \sqrt{J}$, so that choosing $h=1$ and $f = g= 0$, we get  $IJ = gJ + I h$ and $\m J = f J + \m h$.

Suppose that $\calR$ is not a Gorenstein ring and consider the exact sequence  
$$
0 \to \calR \overset{\varphi}{\to} J\calR \to C \to 0
$$
of graded $\calR$-modules with $C \ne (0)$ and $\fkM C = (\xi, \eta)C$ for some homogeneous elements $\xi, \eta$ of $\fkM$. Hence  $\calR_{\fkM}$ is an almost Gorenstein local ring in the sense  of \cite[Definition 3.3]{GTT}. We set $h = \varphi(1) \in J$, $m=\deg \xi$, and $n=\deg \eta$; hence $C = J\calR/\calR h$. Remember that  $h \not\in \m J$, since $\calR_\fkM$ is not a regular local ring (see \cite[Corollary 3.10]{GTT}). If $\min\{m , n\} > 0$, then $\fkM C \subseteq \calR_+ C$, whence $\m C_0 = (0)$  (notice that $[\calR_+ C]_0 = (0)$, as $C=\calR C_0$). Therefore $\m J \subseteq (h)$, so that we have $J =(h)=R$. Thus $\calR h = J \calR$ and  $\calR$ is a Gorenstein ring, which is impossible. Assume $m=0$. If $n=0$, then $\fkM C = \m C$ since $\xi, \eta \in \m$, so that 
$$
C_1 \subseteq \calR_+ C_0 \subseteq \m C
$$
and therefore $C_1 = (0)$ by  Nakayama's lemma. Hence $IJ = Ih$  as $[J\calR]_1 = \varphi (\calR_1)$, which shows $(h)$ is a reduction of $J$, so that $(h) = R=J$. Therefore $\calR$ is a Gorenstein ring, which is impossible. If $n \ge 2$, then because  
$$\fkM{\cdot}J\calR \subseteq \xi{\cdot}J\calR + \eta{\cdot}J\calR  + \calR h,$$
we get $IJ \subseteq \xi IJ + Ih$, whence $IJ = Ih$. This is impossible as we have shown above.  Hence  $n =1$. Let us  write $\eta = gt$ with $g \in I$ and  take $f = \xi$. We then have
$$
\fkM{\cdot}J\calR \subseteq (f, gt){\cdot}J\calR + \calR h,
$$
whence $\m J \subseteq fJ + R h$. Because  $h \not\in \m J$, we get $\m J \subseteq fJ + \m h$, so that $\m J = fJ+\m h$, while $IJ =g J + I h$, because $IJ \subseteq fIJ + gJ + Ih$. This completes the proof of  Theorem \ref{3.3}.
\end{proof}

Let us explore two examples to show how Theorem \ref{3.3} works.

\begin{ex}
Let $S = k[[x, y, z]]$ be the formal power series ring over an infinite field $k$. Let $\n = (x,y,z)$ and choose $f \in \n^2 \setminus \n^3$. We set $R = S/(f)$ and $\m = \n/(f)$. Then  for every  integer $\ell > 0$ the Rees algebra $\calR (\m^\ell)$ of $\m^\ell$ is an almost Gorenstein graded ring and $\mathrm{r}(\calR) = 2\ell + 1$, where $\mathrm{r}(\calR)$ denotes the Cohen-Macaulay type of $\calR$.
\end{ex}

\begin{proof} Since $\e^0_\m(R) = 2$, we have $\m^2 = (a,b)\m$ for some  elements $a, b \in \m$. Let $\ell > 0$ be an integer and set $I = \m^\ell$ and $Q = (a^\ell, b^\ell)$. We then have $I^2 =QI$  and $Q:I = I$, so that $\calR = \calR (I)$ is a Cohen-Macaulay ring  and   $\mathrm{K}_\calR(1) \cong I\calR$  by Proposition \ref{3.1}, whence $\mathrm{r}(\calR) = \mu_R(I) = 2\ell + 1$. Because $\m^{\ell +1} = a\m^\ell + b^\ell \m$ and $Q:I = I = \m^\ell$,  by Theorem \ref{3.3}  $\calR$ is an almost Gorenstein graded ring.
\end{proof}

\begin{ex}
Let $(R,\m)$ be a two-dimensional regular local ring with $\m = (x,y)$. Let $1 \le m \le n$ be integers and set $I = (x^m)+\m^n$. Then $\calR(I)$ is an almost Gorenstein graded ring. 
\end{ex}

\begin{proof} We may assume $m > 1$. 
We set $Q = (x^m,y^n)$ and $J = Q:I$. Then $Q \subseteq I$ and $I^2 = QI$. Since $I = (x^m) +(x^iy^{n-i} \mid 0 \le i \le m-1)$, we get 
\begin{eqnarray*}
J = Q : (x^iy^{n-i} \mid 0 \le i \le m-1)&=& \bigcap_{i=1}^{m-1}\left[(x^m,y^n):x^iy^{n-i}\right] \\
&=& \bigcap_{i=1}^{m-1}(x^{m-i},y^i)\\
&=&\m^{m-1}.
\end{eqnarray*}
Take $f = x \in \m$, $g = x^m \in I$, and $h = y^{m-1}\in J=\m^{m-1}$. We then have $\m J = fJ + \m h$ and $IJ = Ih + gJ$, so that by Theorem \ref{3.3} $\calR (I)$ is an almost Gorenstein graded ring.
\end{proof}

To prove Theorem 1.1 we need a result of J. Verma \cite{V} about joint reductions of integrally closed ideals. Let $(R,\m)$ be a Noetherian local ring. Let $I$ and $J$ be ideals of $R$ and let $a \in I$ and $b \in J$. Then we say that $a, b$ are a {\it joint reduction} of $I, J$ if $aJ + Ib$ is a reduction of $IJ$. Joint reductions always exist (see, e.g., \cite{HS}), if the residue class field of $R$ is infinite. We furthermore have the following.

\begin{thm}[{\cite[Theorem 2.1]{V}}]\label{verma}
Let $(R,\m)$ be a two-dimensional regular local ring. Let $I$ and $J$ be $\m$-primary ideals of $R$. Assume that $a, b$ are a joint reduction of $I$, $J$. Then $IJ = aJ + Ib$, if $I$ and $J$ are integrally closed.
\end{thm}

We are now ready to prove Theorem 1.1.

\begin{proof}[Proof of Theorem 1.1.]  
Let $(R,\m)$ be a two-dimensional regular local ring with infinite residue class field and let $I$ be an $\m$-primary integrally closed ideal in $R$. We choose a parameter ideal $Q$ of $R$ so that $Q \subseteq I$ and $I^2 = QI$ (this choice is possible; see \cite[Appendix 5]{ZS} or \cite{H}). Therefore the Rees algebra $\calR = \calR (I)$ is a Cohen-Macaulay ring (\cite{GS}). Because  $\calR$ is a normal ring (\cite{ZS}), by Corollary \ref{3.2} $J = Q:I$ is an integrally closed ideal in $R$. Consequently, choosing three elements $f \in \m$, $g \in I$, and $h \in J$ so that $f, h$ are a joint reduction of  $\m, J$ and $g, h$ are a joint reduction of $I, J$, we readily get by Theorem \ref{verma} the equalities
$$\m J = fJ + \m g \ \ \text{and} \ \ IJ = gJ + Ih$$
stated in Condition (2) of Theorem \ref{3.3}. Thus $\calR = \calR (I)$ is an almost Gorenstein graded ring.
\end{proof}

We now explore the almost Gorenstein property of the Rees algebras of modules. To state the result we need additional notation. For the rest of this section let $(R, \m)$ be a two-dimensional regular local ring  with infinite residue class field. Let  $M \neq (0)$ be a finitely generated torsion-free $R$-module and assume that $M$ is non-free. Let $(-)^* = \Hom_R(-, R)$. Then $F=M^{**}$ is a finitely generated free $R$-module and we get a canonical exact sequence
$$
0 \to M \overset{\varphi}{\to} F \to C \to 0
$$
of $R$-modules with  $C\neq (0)$ and $\ell_R(C) < \infty$. Let $\operatorname{Sym}(M)$ and $\operatorname{Sym} (F)$ denote the symmetric algebras of $M$ and $F$ respectively and let $\operatorname{Sym} (\varphi) : \operatorname{Sym} (M) \to \operatorname{Sym} (F)$ be the homomorphism induced from $\varphi : M \to F$. 
Then the Rees algebra $\calR(M)$ of $M$ is defined by 
$$
\calR(M) = \Im \left[\Sym(M) \overset{\Sym (\varphi)}{\longrightarrow} \Sym(F)\right]
$$
(\cite{SUV}). Hence $\calR (M) = \Sym (M)/T$ where $T=t(\Sym (M))$ denotes the torsion part of $\Sym (M)$, so that $M = [\calR (M)]_1$ is an $R$-submodule of $\calR (M)$. 
 Let $x \in F$. Then we say that $x$ is integral over $M$, if it satisfies an integral equation 
$$x^n + c_1x^{n-1} + \cdots + c_n = 0$$
in the symmetric algebra $\Sym (F)$ with $n >0$ and $c_i \in M^i$ for each $1 \le i \le n$. Let 
$\overline{M}$ be the set of elements of $F$ which are integral over $M$. Then $\overline{M}$ forms an $R$-submodule of $F$, which is called the integral closure of $M$. We say that $M$ is {\it integrally closed}, if $\overline{M} = M$.

With this notation we have the following.

\begin{cor}\label{3.6} Let $\fkM = \m \calR(M) + \calR(M)_+$ be the unique graded maximal ideal of $\calR(M)$ and suppose that $M$ is integrally closed. Then $\calR(M)_\fkM$ is an almost Gorenstein local ring in the sense of \cite{GTT}.
\end{cor}

\begin{proof} Let $U = R[x_1,x_2, \ldots, x_n]$ be the polynomial ring with sufficiently large $n > 0$ and set $S = U_{\m U}$. We denote by $\n$ the maximal ideal of $S$. Then thanks to \cite[Theorem 3.5]{SUV} and \cite[Theorem 3.6]{HongU}, we can find some elements $f_1, f_2, \ldots, f_{r-1} \in S \otimes_R M$ ($r = \rank_RF$) and an $\n$-primary integrally closed ideal $I$ in $S$, so that $f_1, f_2, \ldots, f_{r-1}$ form a regular sequence in $\calR(S \otimes_R M)$ and
$$
\calR(S \otimes_R M)/(f_1, f_2, \ldots, f_{r - 1}) \cong \calR(I)
$$
as a graded $S$-algebra. Therefore, because  $\calR(I)$ is an almost Gorenstein graded  ring by Theorem 1.1, $S \otimes_R\calR (M) = \calR(S \otimes_R M)$ is an almost Gorenstein graded ring (cf. \cite[Theorem 3.7 (1)]{GTT}). Consequently $\calR(M)_\fkM$ is an almost Gorenstein local ring  by  \cite[Theorem 3.9]{GTT}.
\end{proof}

\section{Almost Gorenstein property in Rees algebras of ideals with linear presentation matrices}

Let $R=k[[x_1, x_2, \ldots, x_d]]~(d\ge 2)$ be the formal power series ring over an infinite field $k$. Let $I$ be a perfect ideal of $R$ with $\operatorname{grade}_R I = 2$,  possessing a linear presentation matrix $\varphi$
$$\ \ \ \ \ 0\to R^{\oplus (n-1)} \overset{\varphi}{\to} R^{\oplus n} \to I \to 0,$$
that is each entry of the matrix $\varphi$ is contained in $\sum_{i=1}^dkx_i$. 
We set $n = \mu_R(I)$ and $\m = (x_1, x_2, \ldots, x_d)$; hence $I = \m^{n-1}$ if $d=2$. In what follows we assume that $n > d$ and that our ideal $I$ satisfies the condition $({\rm G}_d)$ of \cite{AN}, that is $\mu_{R_\fkp} (IR_\fkp) \le \dim R_\fkp$ for every $\p \in \mathrm{V}(I) \setminus \{\m\}$. Then thanks to \cite[Theorem 1.3]{MU} and \cite[Proposition 2.3]{HU}, the Rees algebra $\calR = \calR(I)$ of $I$ is a Cohen-Macaulay ring with $\rma (\calR) = -1$ and $$\rmK_{\calR}(1) \cong \m^{n-d}\calR$$ as a graded $\calR$-module.

We are interested in the question of when $\calR$ is an almost Gorenstein ring. Our answer is the following, which suggests almost Gorenstein Rees algebras might be rare in dimension greater than two.

\begin{thm}\label{4.1}  Let $\fkM$ be the graded maximal ideal of $\calR$.
Then $\calR_{\fkM}$ is an almost Gorenstein local ring if and only if $d=2$.
\end{thm}

\begin{proof}  If $d=2$, then $I = \m^{n-1}$ and so $\calR$ is an almost Gorenstein graded ring (Corollary \ref{3.5}). We assume that $d > 2$ and that $\calR_{\fkM}$ is an almost Gorenstein local ring. The goal is to produce a contradiction.  Let $\Delta_i = (-1)^{i+1} \det \varphi_i$ for each $1 \le i \le n$, where $\varphi_i$ stands for the $(n-1)\times (n-1)$ matrix which is obtained from $\varphi$ by deleting the $i$-th row. Hence $I = (\Delta_1, \Delta_2, \ldots, \Delta_n)$ and the ideal $I$ has a presentation
$$(*)\ \ \ 0 \to R^{\oplus (n-1)} \overset{\varphi}{\to} R^{\oplus n} \overset{\left[\begin{smallmatrix}
\Delta_1&\Delta_2&\cdots&\Delta_n\\
\end{smallmatrix}\right]}
{\longrightarrow} I \to 0.$$
Notice that $\calR$ is not a Gorenstein ring, since $\rmr(\calR)= \mu_R(\m^{n-d})= \binom{n-1}{d-1} >1$. We set $A = \calR_\fkM$  and  $\n = \fkM A$; hence $\mathrm{K}_A = [\mathrm{K}_\calR]_\fkM$. We take an  exact sequence
$$
0 \to A \overset{\phi}{\to} \mathrm{K}_A \to C \to 0
$$
of $A$-modules such that $C \neq (0)$ and $C$ is an Ulrich $A$-module. Let $f = \phi(1)$. Then $f \not\in \n \mathrm{K}_A$ by \cite[Corollary 3.10]{GTT} and we get the exact sequence
$$(**) \ \ \ \ \ 
0 \to \n f \to \n \rmK_A \to \n C \to 0.
$$
Because   
$
\n C = (f_1, f_2, \ldots, f_d) C
$
for some $f_1, f_2, \ldots, f_d \in \n$ (\cite[Proposition 2.2]{GTT}) and $\mu_A(\n) = d+n$, we get by the exact sequence ($**$) that 
$$
\mu_{\calR}(\fkM \rmK_{\calR}) 
=   \mu_{A}(\fkn \rmK_A) 
\le    (d + n) +d {\cdot}\left(\mathrm{r}(A)-1\right) 
=  d \binom{n-1}{d-1} + n,
$$
 while 
$$
\mu_{\calR}(\fkM \rmK_{\calR})  = \mu_R(\m^{n-d+1}) + \mu_R(\m^{n-d}I) = \binom{n}{d-1} + \mu_R(\m^{n-d}I)
$$ 
since $\fkM =(\m, It)\calR$ and $\mathrm{K}_\calR(1) = \m^{n-d}\calR$.
Consequently we have
$$
\mu_R(\m^{n-d} I ) \le d \binom{n-1}{d-1} + n - \binom{n}{d-1}.
$$
To estimate the number $\mu_R(\m^{n-d} I )$ from below, we consider  the homomorphism $$\psi : \m^{n-d}\otimes_R I \to \m^{n-d} I$$ defined by  $x\otimes y\mapsto xy $ and set $X= \Ker \psi$.
Let $x \in X$ and write $x = \sum_{i=1}^dx_i\otimes \Delta_i$ with $x_i \in \m^{n-d}$. Then since $\sum_{i=1}^dx_i\Delta_i= 0$ in $R$ and since every entry of the matrix $\varphi$ is linear, the presentation $(*)$ of $I$ guarantees the existence of   elements  $y_j \in \m^{n-d-1}$~$(1 \le j \le n-1)$ such that
$$
\begin{pmatrix}
x_1 \\
x_2 \\
\vdots \\
x_n
\end{pmatrix}= \varphi
\begin{pmatrix}
y_1 \\
y_2 \\
\vdots \\
y_{n-1}
\end{pmatrix}.
$$
Hence $X$ is a homomorphic image of  $[\m^{n-d-1}]^{\oplus (n-1)}$ and therefore 
$$
\mu_R(X) \le (n-1) \binom{n-2}{d-1}
$$
in the exact sequence 
$$0 \to X \to \m^{n-d}\otimes_RI \to \m^{n-d}I \to 0.$$
Consequently 
$$
\mu_R(\m^{d-n}I) \ge n \binom{n-1}{d-1} - (n-1)\binom{n-2}{d-1},
$$
so that 
\begin{eqnarray*}
0 &\le& \left[d\binom{n-1}{d-1} + n - \binom{n}{d-1}\right] -\left[n \binom{n-1}{d-1} - (n-1) \binom{n-2}{d-1}\right] \\
   &=& \left[(d-n) \binom{n-1}{d-1} + (n-1) \binom{n-2}{d-1} \right] + \left[ n - \binom{n}{d-1} \right] \\
   &=& n - \binom{n}{d-1} < 0
\end{eqnarray*}
which is the required contradiction. Thus  $A$ is not an almost Gorenstein local ring, if $d>2$.
\end{proof}

Before closing this section, let us note one concrete example.

\begin{ex}
Let $R = k[[x,y,z]]$ be the formal power series ring over an infinite  field $k$. We set $I =(x^2y,y^2z,z^2x,xyz)$ and $Q = (x^2y, y^2z, z^2x)$. Then $Q$ is a minimal reduction of $I$ with $\mathrm{red}_Q(I)=2$. The ideal $I$ has a presentation of the form 
$$0 \to \R^{\oplus 3} \overset{\varphi}{\to} R^{\oplus 4} \to I \to 0$$
with $\varphi = \left(\begin{smallmatrix}
x&0&0\\
0&y&0\\
0&0&z\\
y&z&x
\end{smallmatrix}\right)$ and it is direct to check that $I$ satisfies all the conditions required for Theorem \ref{4.1}. Hence Theorem \ref{4.1} shows that  $\calR (I)$ cannot be an almost Gorenstein graded ring, while $Q$ is not a perfect ideal of $R$ but its Rees algebra  $\calR(Q)$ is an almost Gorenstein graded ring with $\mathrm{r}(\calR) = 2$; see \cite{Kamoi}. 
\end{ex}

\section{The Rees algebras of socle ideals $I = (a,b):\m$}
Throughout this section let $(R,\m)$ denote a two-dimensional regular local ring and let $Q=(a,b)$ be a parameter ideal of $R$. We set $I = Q:\m$, $\calR = \calR(I)$, and $\fkM = \m \calR+ \calR_+$. In this section we are interested in the question of when $\calR_\fkM$ is an almost Gorenstein local ring. If $Q \not\subseteq \m^2$, then $Q:\m$ is a parameter ideal of $R$ (or the whole ring $R$) and $\calR$ is a Gorenstein ring.

Assume that $Q \subseteq \m^2$. Hence $I^2 = QI$ (see e.g., \cite{W}) with $\mu_R(I)=3$. We write $I = (a,b,c)$. Then since $xc, yc \in Q$, we get equations
$$f_1a + f_2b + xc = 0 \ \ \ \text{and}\ \ \ g_1a + g_2b + yc = 0$$
with $f_i, g_i \in \m$~($i= 1,2$). With this notation we have the following.

\begin{thm}\label{thm1}
If $(f_1, f_2, g_1,g_2) \subseteq \m^2$, then $\calR_\fkM$ is not an almost Gorenstein local ring. 
\end{thm}

We divide the proof of Theorem \ref{thm1} into a few steps. Let us  begin with the following. 

\begin{lem}\label{lem1} Let $\Bbb{M} = \begin{pmatrix}
f_1&f_2&x\\
g_1&g_2&y\\
\end{pmatrix}$. Then 
$R/I$ has  a minimal free resolution
$$0 \to R^{\oplus 2} \overset{{}^t\Bbb{M}}
{\longrightarrow} R^{\oplus 3} \overset{[\begin{smallmatrix}
a&b&c\\
\end{smallmatrix}]}
{\longrightarrow} R \to R/I \to 0.$$ 
\end{lem}

\begin{proof}
Let 
$\mathbf{f} = \begin{pmatrix}
f_1\\
f_2\\
x
\end{pmatrix}$ and 
$\mathbf{g} = \begin{pmatrix}
g_1\\
g_2\\
y
\end{pmatrix}$. Then $\mathbf{f}, \mathbf{g} \in \m{\cdot}R^{\oplus 3}$. As $\mathbf{f}, \mathbf{g}~\mod~\m^2{\cdot}R^{\oplus 3}$ are linearly independent over $R/\m$, the complex 
$$0 \to R^{\oplus 2} \overset{{}^t\Bbb{M}}
{\longrightarrow} R^{\oplus 3} \overset{[\begin{smallmatrix}
a&b&c\\
\end{smallmatrix}]}
{\longrightarrow} R \to R/I \to 0$$
is exact and gives rise to a minimal free resolution of $R/I$.
\end{proof}

Let  $\calS =R[X,Y,Z]$ be the polynomial ring and let $\varphi : \calS \to \calR =R[It]$~($t$ an indeterminate) be the $R$-algebra map defined by $\varphi (X) = at$, $\varphi (Y) = bt$, and $\varphi (Z) = ct$. Let $K= \Ker \varphi$.  Since $c^2 \in QI$, we have  a relation of the form
$$c^2 = a^2f +b^2g + abh + bci + caj$$
with $f,g,h, i,j \in R$. We set 
\begin{eqnarray*}
F &=& Z^2 - \left(fX^2 + gY^2 + hXY + iYZ + jZX \right),\\
G&=& f_1X+f_2Y+xZ,\\
H&=&g_1X + g_2Y + yZ. 
\end{eqnarray*}
Notice that $F\in \calS_2$ and $G,H \in \calS_1$.

\begin{prop}\label{lem2} $\calR$ has a minimal graded free resolution of the form
$$0 \to \calS(-2)\oplus\calS(-2) \overset{{}^t\Bbb{N}}{\longrightarrow} \calS(-2)\oplus \calS(-1) \oplus \calS(-1) \overset{[\begin{smallmatrix}
F&G&H\\
\end{smallmatrix}]}{\longrightarrow} \calS \to \calR \to0,$$ so that the graded canonical module $\mathrm{K}_\calR$ of $\calR$ has a presentation
$$ \calS(-1)\oplus \calS(-2) \oplus \calS(-2)  \overset{\Bbb{N}}{\longrightarrow} \calS(-1)\oplus \calS(-1)\to \mathrm{K}_\calR \to 0.$$
\end{prop}

\begin{proof}
We have $K= \calS K_1 + (F)$ (remember that $I^2 = QI$ and $c^2 \in QI$). Hence  $\calR$ has a minimal graded free resolution of the form
$$(*)\ \ \ \ 0 \to \calS(-m)\oplus\calS(-\ell) \overset{{}^t\Bbb{N}}{\longrightarrow} \calS(-2)\oplus \calS(-1) \oplus \calS(-1) \overset{[\begin{smallmatrix}
F&G&H\\
\end{smallmatrix}]}{\longrightarrow} \calS \to \calR \to 0$$
with $m, \ell \ge 1$. We take the $\calS(-3)$-dual of the resolution ($*$). Then as $\mathrm{K}_\calS = \calS(-3)$, we get the presentation 
$$ \calS(-1)\oplus \calS(-2) \oplus \calS(-2)  \overset{\Bbb{N}}{\longrightarrow} \calS(m-3)\oplus \calS(\ell -3)\to \mathrm{K}_\calR \to 0$$
of the canonical module $\mathrm{K}_\calR$ of $\calR$. Hence $m, \ell \le 2$ because $\mathrm{a}(\calR) = -1$. Assume that  $m = 1$. Then the matrix ${}^t\Bbb{N}$ has the form ${}^t\Bbb{N} = 
\begin{pmatrix}
0&\beta_1\\
\alpha_2&\beta_2\\
\alpha_3&\beta_3
\end{pmatrix}$ with $\alpha_2, \alpha_3 \in R$. We have $\alpha_2G + \alpha_3 H = 0$, or equivalently
$\alpha_2 \begin{pmatrix}
f_1\\
f_2\\
x
\end{pmatrix} + \alpha_3\begin{pmatrix}
g_1\\
g_2\\
y
\end{pmatrix}= \mathbf{0}$, whence $\alpha_2 = \alpha_3 = 0$ by Lemma \ref{lem1}. This is impossible, whence $m= 2$. We similarly have $\ell = 2$ and the assertion follows. 
\end{proof}

We are now ready to prove Theorem \ref{thm1}. 

\begin{proof}[Proof of Theorem \ref{thm1}] Let $\Bbb{N}$ be the matrix given by  Proposition  \ref{lem2} and write   $N = \begin{pmatrix}
\alpha&F_1&F_2\\
\beta&G_1&G_2
\end{pmatrix}$. Then Proposition \ref{lem2} shows that $F_i,G_i \in \calS_1$~($i=1,2$) and $\alpha, \beta \in \m$. We write  $F_i = \alpha_{i1}X + \alpha_{i2}Y + \alpha_{i3}Z$ and  $G_i = \beta_{i1}X + \beta_{i2}Y + \beta_{i3}Z$ with $\alpha_{ij}, \beta_{ij} \in R$. Let $\Delta_j$ denote the determinant of the matrix obtained by deleting the $j$-th column from $\Bbb{N}$. Then by the theorem of Hilbert-Burch we have $G= -\varepsilon \Delta_2$ and $H=\varepsilon \Delta_3$ for some unit $\varepsilon$ of $R$, so that 
$$(**) \ \ \ 
\begin{pmatrix}
f_1\\
f_2\\
x
\end{pmatrix}
= (\varepsilon\beta) \begin{pmatrix}
\alpha_{21}\\
\alpha_{22}\\
\alpha_{23}
\end{pmatrix}
-(\varepsilon\alpha) \begin{pmatrix}
\beta_{21}\\
\beta_{22}\\
\beta_{23}
\end{pmatrix}\ \ \ \text{and} \ \ 
\begin{pmatrix}
g_1\\
g_2\\
y
\end{pmatrix}
= (\varepsilon\alpha) \begin{pmatrix}
\beta_{11}\\
\beta_{12}\\
\beta_{13}
\end{pmatrix}
- (\varepsilon\beta) \begin{pmatrix}
\alpha_{11}\\
\alpha_{12}\\
\alpha_{13}
\end{pmatrix}.$$
Hence $$x = \varepsilon\left(\beta \alpha_{23} - \alpha \beta_{23}\right) \ \ \ \text{and} \ \ \ y = \varepsilon\left(\alpha \beta_{13}- \beta \alpha_{13}\right),$$
which shows   $(x,y)=(\alpha, \beta)=\m$,  because  $(x,y) \subseteq (\alpha, \beta) \subseteq \m$. Therefore if $(f_1, f_2, g_1,g_2) \subseteq \m^2$, then equations ($**$) above show $\alpha_{ij}, \beta_{ij} \in \m$ for all $i,j =1,2$, whence 
$$N \equiv \begin{pmatrix}
\alpha&\alpha_{13}Z&\alpha_{23}Z\\
\beta&\beta_{13}Z&\beta_{23}Z
\end{pmatrix} \ \mod~\fkN^2$$
where $\fkN = \m \calS +\calS_+$ is the graded maximal ideal of $\calS$. We set $B = \calS_\fkN$. Then it is clear that after any elementary row and column operations the matrix $\Bbb{N}$ over the regular local ring $B$ of dimension 5 is not equivalent to a matrix of the form
$$
\begin{pmatrix}
\alpha_1&\alpha_2&\alpha_3\\
\beta_1&\beta_2&\beta_3\\
\end{pmatrix} 
$$ with $\alpha_1, \alpha_2, \alpha_3$ a part of a regular system of parameters of $B$. Hence by \cite[Theorem 7.8]{GTT} $\calR_\fkM$ cannot be an almost Gorenstein local ring.
\end{proof}

As a consequence of Theorem \ref{thm1} we get the following.

\begin{cor}
Suppose that $Q=(a,b) \subseteq \m^3$. Then $\calR_\fkM$ is not an almost Gorenstein local ring. 
\end{cor}

\begin{proof}
We write $\begin{pmatrix}
a\\
b
\end{pmatrix} = \begin{pmatrix}
f_{11}&f_{12}\\
f_{21}&f_{22}
\end{pmatrix}
\begin{pmatrix}
x\\
y
\end{pmatrix}$ with $f_{ij} \in \m^2$~($i,j=1,2$) and set $c = \det \begin{pmatrix}
f_{11}&f_{12}\\
f_{21}&f_{22}
\end{pmatrix}$. Then $Q : c = \m$ and $Q:\m = Q + (c)$. We have
$$(-f_{22})a+f_{12}b +cx = 0 \ \ \text{and} \ \  f_{21}a + (-f_{11})b+cy = 0.$$  Hence  by Theorem \ref{thm1} $\calR_\fkM$ is not an almost Gorenstein local ring. \end{proof}

\begin{cor}\label{badex}
Let $m \ge n \ge 2$ be integers and set $Q = (x^m, y^n)$. Then $\calR$ is an almost Gorenstein graded ring if and only if $n =2$.
\end{cor}

\begin{proof} Suppose $n = 2$. Then $K = \begin{pmatrix}
x&Z&X\\
y&x^{m-2}Y&Z
\end{pmatrix}
$. Hence by  \cite[Theorem 7.8]{GTT} $\calR$ is an almost Gorenstein graded ring. Conversely, suppose that $\calR$ is an almost Gorenstein graded ring. Then $n =2$ by Theorem \ref{thm1}, because $\calR_\fkM$ is an almost Gorenstein local ring.
\end{proof}





\begin{thebibliography}{20}
\bibitem{AN}
{\sc M. Artin and M. Nagata}, Residual intersections in Cohen-Macaulay rings, {\em J. Math. Kyoto Univ.}, {\bf 12} (1972), 307--323.

\bibitem{BF}
{\sc V. Barucci and R. Fr\"{o}berg}, One-dimensional almost Gorenstein rings, {\em J. Algebra}, {\bf 188} (1997), no. 2, 418--442.

\bibitem{BHU}
{\sc J. P. Brennan, J. Herzog and B. Ulrich}, Maximally generated maximal Cohen-Macaulay modules, {\em Math. Scand.}, {\bf 61} (1987), no. 2, 181--203.

\bibitem{GMP}
{\sc S. Goto, N. Matsuoka and T. T. Phuong}, Almost Gorenstein rings, {\em J. Algebra}, {\bf 379} (2013), 355--381.

\bibitem{GS}
{\sc S. Goto and Y. Shimoda}, On the Rees algebras of Cohen-Macaulay local rings, Commutative algebra (Fairfax, Va., 1979), 201--231, Lecture Notes in Pure and Appl. Math., 68, Dekker, New York, 1982.

\bibitem{GTT}
{\sc S. Goto, R. Takahashi and N. Taniguchi}, Almost Gorenstein rings -towards a theory of higher dimension, {\em J. Pure Appl. Algebra}, {\bf 219} (2015), 2666--2712.

\bibitem{GTY1}
{\sc S. Goto, N. Taniguchi and K.-i. Yoshida}, The almost Gorenstein Rees algebras of $p_g$-ideals, Preprint 2015.

\bibitem{GTY2}
{\sc S. Goto, N. Matsuoka, N. Taniguchi and K.-i. Yoshida}, The almost Gorenstein Rees algebras of parameters, Preprint 2015.

\bibitem{H}
{\sc C. Huneke}, Complete ideals in two-dimensional regular local rings, {\em MSRI Publications}, {\bf 15} (1989), 325--338.

\bibitem{HongU}
{\sc J. Hong and B. Ulrich},  Specialization and Integral Closure, {\em J. London Math. Soc.}, {\bf 90} (3) (2014), 861--878.

\bibitem{HU}
{\sc C. Huneke and B. Ulrich}, Residual intersections, {\em J. reine angew. Math.}, {\bf 390} (1988), 1--20.

\bibitem{Ikeda}
{\sc S. Ikeda}, On the Gorensteinness of Rees algebras over local rings, {\em Nagoya Math. J.}, {\bf 102} (1986), 135--154.

\bibitem{Kamoi} 
{\sc Y. Kamoi}, Gorenstein Rees algebras and Huneke-Ulrich ideals, in preparation.

\bibitem{MU}
{\sc S. Morey and B. Ulrich}, Rees algebras of ideals with low codimension, {\em Proc. Amer. Math. Soc.}, {\bf 124} (1996), 3653--3661.

\bibitem{HS}
{\sc I. Swanson and C. Huneke}, Integral Closure of Ideals, Rings, and Modules, {\em Cambridge University Press}, 2006.

\bibitem{SUV}
{\sc A. Simis, B. Ulrich and W. V. Vasconcelos}, Rees algebras of modules, {\em Proc. London Math. Soc.}, {\bf 87} (3) (2003), 610--646.

\bibitem{T}
{\sc R. Takahashi}, On G--regular local rings, {\em Comm. Algebra}, {\bf 36} (2008), no. 12, 4472--4491.

\bibitem{U} 
{\sc B. Ulrich}, Ideals having the expected reduction number, {\em Amer. J. Math.}, {\bf 118} (1996), no. 1, 17--38.

\bibitem{V}
{\sc J. K. Verma}, Joint reductions and Rees algebras, {\em Math. Proc. Camb. Phil. Soc.}, {\bf 109} (1991), 335--342.

\bibitem{W} 
{\sc H.-J. Wang}, Links of symbolic powers of prime ideals, {\em Math. Z.}, {\bf 256} (2007), 749--756.


\bibitem{ZS}
{\sc O. Zariski and P. Samuel}, Commutative Algebra Volume II, {\em Springer}, 1960.

\end{thebibliography}
\end{document}